\documentclass{scrartcl}
\usepackage[english]{babel}
\usepackage{egen} %Personal macros
\usepackage[all,cmtip]{xy} %For commutative diagrams
\usepackage{hyperref}
\usepackage{enumerate} %In order to get (i)-(iii)-lists.

\theoremstyle{plain}
\newtheorem{nthm}{Theorem}[section]
\newtheorem{nprop}[nthm]{Proposition}
\newtheorem{nlemma}[nthm]{Lemma}

\theoremstyle{definition}

\newtheorem{nquestion}[nthm]{Question}

\newcommand{\mpp}{\mu_+[\pt]}
\newcommand{\mmp}{\mu_-[\pt]}
\newcommand{\mpc}{\mu_+[C]}
\newcommand{\mmc}{\mu_-[C]}
\newcommand{\mcpp}{\mu^{\mathrm{c}}_+[\pt]}
\newcommand{\mcmp}{\mu^{\mathrm{c}}_-[\pt]}
\newcommand{\mcpc}{\mu^{\mathrm{c}}_+[C]}
\newcommand{\mcmc}{\mu^{\mathrm{c}}_-[C]}
\newcommand{\vcc}{V^{\mathrm{c}}(C)}
\newcommand{\wc}{W^{\mathrm{c}}}
\newcommand{\gr}{\mathrm{gr}}

\title{Homology of Hilbert schemes of points on a locally planar curve}
\author{Jørgen Vold Rennemo}
\date{}
\begin{document}
\maketitle
\begin{abstract}
Let $C$ be a proper, integral, locally planar curve, and consider its Hilbert schemes of points $C^{[n]}$. We define 4 creation/annihilation operators acting on the rational homology groups of these Hilbert schemes and show that the operators satisfy the relations of a Weyl algebra. The action of this algebra is similar to that defined by Grojnowski and Nakajima for a smooth surface.

As a corollary, we compute the cohomology of $C^{[n]}$ in terms of the cohomology of the compactified Jacobian of $C$ together with an auxiliary grading on the latter. This recovers and slightly strenghtens a formula recently obtained in a different way by Maulik and Yun and independently Migliorini and Shende.
\end{abstract}

\section{Introduction}
\label{sec:introduction}
Let $C$ be a proper, integral, complex curve with planar singularities. Denote by $C^{[n]}$ the Hilbert scheme of length $n$ subschemes of $C$. Let $J$ be the compactified Jacobian, i.e.\ the space of torsion free sheaves on $C$ with rank 1 and degree 0. These spaces are related by the Abel--Jacobi morphism $AJ : C^{[n]} \to J$, which sends a subscheme $Z$ to the sheaf $\cI_Z \otimes \cO(x)^{\otimes n}$, where $x\in C$ is a chosen nonsingular point. Under our assumptions on $C$, both $C^{[n]}$ and $J$ are reduced and irreducible with l.c.i.\ singularities \cite{AltmanIarrobinoKleiman77, briancon81}. 

Let $g$ be the arithmetic genus of $C$. For $n \ge 2g-1$ the map $AJ$ is a $\PP^{n-g}$-bundle (see \cite{altman_compactifying_1980}), so the rational homology group $H_*(C^{[n]})$ is determined up to isomorphism by $H_*(J)$. The formula below extends this by expressing $H_*(C^{[n]})$ in terms of $H_*(J)$ even for $n < 2g - 1$.

In order to state the result, we will define a new grading on $H_*(J)$, with the $m$-th graded piece denoted $D_mH_*(J)$. This $D$-grading combines with the homological grading to give a bigrading, and we have $D_mH_*(J) = 0$ unless $0 \le m \le 2g$. We then have the following formula.
\begin{nprop}
\label{thm:MacdonaldsFormula}
There is an isomorphism of homologically graded vector spaces
\[
H_*(C^{[n]}) \cong \bigoplus_{m\le n} D_mH_*(J) \otimes \Sym^{n-m}(\QQ \oplus \QQ[2]).
\] 
\end{nprop}
Here $\QQ[2]$ denotes the space $\QQ$ with homological degree 2. A very similar statement was recently shown by Maulik and Yun \cite{maulik11} and Migliorini and Shende \cite{migliorini11}. See Section \ref{sec:relationtoexistingwork} for a discussion of how these papers relate to this one.

\subsection{Algebra action}
Proposition \ref{thm:MacdonaldsFormula} will be obtained as a corollary of our main result, which we now describe. Consider the vector space
\[
V(C) := \bigoplus_{n \ge 0} H_*(C^{[n]}).
\]
We shall define two pairs of creation and annihilation operators acting on $V(C)$. 

The first pair is denoted $\mu_{\pm}[\pt] : H_*(C^{[n]}) \to H_{*-1\pm 1}(C^{[n\pm 1]})$ and corresponds to adding or removing a fixed nonsingular point $x \in C$. Indeed, any such $x$ induces an inclusion $i : C^{[n]} \into C^{[n+1]}$ by letting $\cI_{i(Z)} = \cI_Z \cdot \cI_x$ for every $Z \in C^{[n]}$. We then take $\mpp = i_*$ and $\mmp = i^!$, where $i^!$ is the intersection pullback map.

The second pair is denoted $\mu_{\pm}[C] : H_*(C^{[n]}) \to H_{*+1\pm 1}(C^{[n\pm 1]})$, and the operators are correspondences induced by the diagram
\[
\xymatrix{&C^{[n,n+1]} \ar@{->}[dl]_p \ar@{->}[dr]^q& \\
C^{[n]}&&C^{[n+1]},}
\]
that is $\mpc = q_{*}p^{!}$ and $\mmc = p_{*}q^{!}$.
Here $C^{[n,n+1]}$ is the flag Hilbert scheme of pairs $(Z,Z^\pr) \in C^{[n]}\times C^{[n+1]}$ such that $Z \subset Z^\pr$. 
The fact that the Gysin maps $p^{!}, q^{!}$ are well defined is nontrivial, since all three schemes are in general singular. In particular, the definition depends on the assumption that $C$ is locally planar; see Section \ref{sec:definitionofoperators}.

The main result of this paper is the following.
\begin{nthm}\leavevmode
\label{thm:MainTheorem}
\begin{enumerate}[(i)]
\item \label{thm:MainTheorem:CommRels} The operators $\mu_{\pm}[\pt], \mu_\pm[C] \in \mathrm{End}(V(C))$ satisfy the commutation relations
\[
[\mmp,\mpc] = [\mmc, \mpp] = \id,
\]
and all other pairs of operators commute.
\item \label{thm:MainTheorem:StructureResult} Let $W = \ker \mmp \cap \ker \mmc$. Then the natural map
\[
W \otimes \QQ[\mpp, \mpc] \to V(C)
\]
is an isomorphism.
\item \label{thm:MainTheorem:WEqualsJ} The Abel--Jacobi pushforward map $AJ_* : V(C) \to H_*(J)$ induces an isomorphism $W \cong H_*(J)$.
\end{enumerate}
\end{nthm}
Point (\ref{thm:MainTheorem:CommRels}) can be rephrased as saying that the subalgebra of $\mathrm{End}(V(C))$ generated by $\mu_{\pm}[\pt], \mu_\pm[C]$ is isomorphic to the Weyl algebra $\CC[x_{1},x_{2},\partial_{1},\partial_{2}]$.

Note that $V(C)$ is naturally bigraded by taking the $(i,n)$-th homogeneous piece to be $H_i(C^{[n]})$. The four operators are bihomogeneous, so the space $W$ in the theorem inherits a bigrading, and so by part (\ref{thm:MainTheorem:WEqualsJ}) we get an induced bigrading on $H_*(J)$. We let $D_nH_i(J)$ denote the $(i,n)$-th homogeneous part of $H_*(J)$. Restricting the isomorphism of (\ref{thm:MainTheorem:StructureResult}) to a single $H_*(C^{[n]})$ then gives Proposition \ref{thm:MacdonaldsFormula}.

\subsection{On the proof}
Assuming the commutation relations of Theorem \ref{thm:MainTheorem} (\ref{thm:MainTheorem:CommRels}), the proof of part (\ref{thm:MainTheorem:StructureResult}) is a matter of elementary algebra.
The proof of (\ref{thm:MainTheorem:WEqualsJ}) is then quite easy, using the fact that for large $n$ the map $C^{[n]} \to J$ is a projective space bundle \cite{altman_compactifying_1980}.

Finally, for checking the commutation relations of (\ref{thm:MainTheorem:CommRels}), the idea is the following. The operators can all be thought of as correspondences. If the $C^{[n]}$ were smooth, we could apply the usual composition formula for correspondences, and so reduce the calculation of each commutator to computing a specific class in $H_*(C^{[n]}\times C^{[n^{\pr}]})$, with $n^{\pr} \in \{n-2, n, n+2\}$.

The idea for circumventing the non-smoothness of the $C^{[n]}$ is to embed $C$ in an algebraic family $\cC \to B$ over a smooth base $B$, such that the relative Hilbert schemes $\cC^{[n]}\to B$ are nonsingular for all $n$. That this is possible follows from the fact that $C$ is locally planar, as was shown by Shende \cite[Cor.\ 15]{shende12}. Given such a family, we may compose correspondences in the family, compute the commutators (this is possible by the nonsingularity of $\cC^{[n]}$), and finally restrict to the fibre $C^{[n]}$.
\subsection{Variants}
The main theorem has natural variants in cohomology and Chow homology:
\subsubsection{Cohomology}
\label{sec:remarkoncohomology}
Since we are working with $\QQ$-coefficients, we may dualise every vector space and consider cohomology instead of homology. Let
\[
V^{\text{c}}(C) = \bigoplus_{i,n\ge 0} H^i(C^{[n]}, \QQ).
\]
We let the operators $\mu_\pm^{\text{c}}$ acting on cohomology be defined by dualising, i.e.\ by $\mu^{\text{c}}_\pm[\pt] = \mu_\mp[\pt]^*$ and $\mu_\pm^{\text{c}}[C] = \mu_\mp[C]^*$.

Then from Theorem \ref{thm:MainTheorem} we easily get the following cohomological version. 
\begin{nthm}
\label{thm:MainTheoremCohomology}
\leavevmode
\begin{enumerate}[(i)]
\item \label{thm:MainTheoremCohomology:CommRels} The operators $\mu^{\mathrm{c}}_{\pm}[\pt], \mu^{\mathrm{c}}_\pm[C] \in \mathrm{End}(\vcc)$ satisfy the commutation relations
\[
[\mcmp,\mcpc] = [\mcmc, \mcpp] = \id,
\]
and all other pairs of operators commute.
\item \label{thm:MainTheoremCohomology:StructureResult} Let $\wc = \vcc/(\im \mcpp + \im \mcpc)$. Then the natural maps
\[
\ker \mcmp \cap \ker \mcmc \to \wc
\]
and
\[
(\ker \mcmp \cap \ker \mcmc) \otimes \QQ[\mcpp, \mcpc] \to V^{c}(C)
\]
are isomorphisms.
\item \label{thm:MainTheoremCohomology:WEqualsJ} The Abel--Jacobi pullback map $AJ^* : H^*(J) \to H^*(C^{[n]})$ induces an isomorphism $H_*(J) \cong \wc$.
\end{enumerate}
\end{nthm}

The natural bigrading on $V^{\text{c}}(C)$ induces a bigrading on $W^{\text{c}}$, and hence a bigrading on $H^*(J)$, which we write as $H^*(J) = \oplus_{i,n}D_nH^i(J)$. As in the case of homology, we recover every $H^*(C^{[n]})$ from the data of $H^*(J)$ with this $D$-grading, i.e.\  
\begin{equation}
\label{MacdonaldFormulaCohomology}
H^*(C^{[n]}) \cong \bigoplus_{m\le n}D_mH^*(J)\otimes\Sym^{n-m}\left(\QQ \oplus \QQ[-2]\right).
\end{equation}

The following question seems natural.
\begin{nquestion}
Is the cup product on $H^*(J)$ homogeneous with respect to the $D$-grading?
\end{nquestion}

\subsubsection{Chow homology}
Instead of the homology groups $H_*(C^{[n]})$ and $H_*(J)$ we may work with Chow homology groups $A_*(C^{[n]})$ and $A_*(J)$ (with rational coefficients). The operators $\mu_{\pm}[\pt]$ and $\mu_\pm[C]$ can still be defined in this setting, and Theorem \ref{thm:MainTheorem} then holds. The proof is the same as in the case of singular homology, and we shall only indicate the changes necessary at the few places where these occur.

Note that in this setting the operators $\mu_\pm[\pt]$ will in general depend on the particular point $x \in C$ chosen for the definition of $C^{[n]} \into C^{[n+1]}$.

\subsection{Applications to curve counting and BPS numbers}
The present work is related to considerations in curve counting on Calabi--Yau 3-folds. See also the introduction to \cite{migliorini11} or the survey paper \cite{PandharipandeThomas11} for background on these curve counting theories.

Under our assumptions on the curve $C$, Pandharipande and Thomas \cite[App.\ B]{PandharipandeThomas10} show that there are integers $n_g$ such that
\begin{equation}
\label{eqn:BPSDefinition}
q^{1-g(C)}\sum_{n=0}^\infty \chi(C^{[n]})q^n = \sum_{g = g(\wt{C})}^{g(C)} n_g\left(\frac{q}{(1-q)^2}\right)^{1-g(C)}.
\end{equation}
Here $g(C)$ and $g(\wt{C})$ are the arithmetic and geometric genera of $C$, respectively. If $C$ lies in a Calabi--Yau 3-fold $X$, then one may in certain good cases interpret the $n_g$ as the contribution of $C$ to the BPS invariant $n_{g,[C]}$ of Gopakumar and Vafa, see \cite{PandharipandeThomas10}.

In Gopakumar and Vafa's original proposal \cite{GopakumarVafa98, GopakumarVafa982} the BPS invariants $n_{g,[C]}$ of a Calabi--Yau 3-fold $X$ are computed from the cohomology of the space of pure 1-dimensional sheaves on $X$. For a single curve $C$, this computation suggests the following alternative way of defining the contribution of $C$ to $n_{g,[C]}$: The cohomology $H^*(J)$ should in some sense split as the direct sum of cohomologies $H^*(T^{2g})$ for different $g$, where $T^{2g}$ is real $2g$-dimensional torus.
The contribution of $C$ to $n_{g,[C]}$ should then be the number of copies of $H^*(T^{2g})$ appearing in the decomposition. 

Formula \eqref{MacdonaldFormulaCohomology} gives one way of making this precise, as follows. The right hand side of \eqref{eqn:BPSDefinition} is a rational function invariant under $q \mapsto q^{-1}$, hence the left hand side is as well. Let $\chi(D_nH^*(J)) = \dim D_nH^{\text{even}}(J) - \dim D_nH^{\text{odd}}(J)$. Applying \eqref{MacdonaldFormulaCohomology} one can then check that the Laurent polynomial
\[
q^{-g(C)}\sum_{n=0}^{2g(C)}\chi(D_nH^*(J))q^n
\]
is invariant under $q \mapsto q^{-1}$ as well.\footnote{The symmetry of this polynomial can be refined to an isomorphism $D_nH^k(J) \cong D_{2g(C)-n}H^{k+2g(C)-2n}(J)$. This follows from the relation between the $D$-grading and the perverse filtration on $H^*(J)$ (Prop. \ref{thm:DEqualsP}) and the relative hard Lefshetz theorem applied to the perverse filtration, see \cite[2.16]{maulik11}.}

Thinking of $(q^{-1}-2+q)^g$ as the shifted Poincaré polynomial of $T^{2g}$, it is then reasonable to define the contribution $n^\prime_g$ of $C$ to $n_{g,[C]}$ by
\[
q^{-g(C)}\sum_{n=0}^{2g(C)}\chi(D_nH^*(J))q^n = \sum_{g = 0}^{2g(C)} n_g^{\pr}(q^{-1}-2+q)^g.
\]
From \eqref{MacdonaldFormulaCohomology} we then easily get the following proposition.
\begin{nprop}
The two definitions of the contribution of $C$ to the BPS number $n_{g,[C]}$ agree, i.e.\ we have
\[
n_g = n_g^\pr
\]
for all $g$.
\end{nprop}

\subsection{Relation to existing work}
\label{sec:relationtoexistingwork}
The results in this paper are motivated by the recent work of Maulik and Yun \cite{maulik11} and Migliorini and Shende \cite{migliorini11}. In these papers $H^*(J)$ is endowed with a certain perverse filtration $P$, and the $P$-graded space $\gr^P_*H^*(J)$ then recovers $H^*(C^{[n]})$ in the same way that our $D$-graded $H^*(J)$ recovers $H^*(C^{[n]})$. In Section \ref{sec:GradingRefinesFiltration}, we show that the grading $D$ is in fact a splitting of the filtration $P$.

This filtration $P$ arises in a completely different way to our $D$-grading. Consider a deformation family $\cC \to B$ such that the relative compactified Jacobian $f: \cJ \to B$ is nonsingular. Then $Rf_*(\QQ_\cJ)\in D^b_c(B)$ has a filtration induced by the perverse $t$-structure on $D^b_c(B)$, which restricts to give the filtration $P$ on $H^*(J)$. 
The main result of \cite{maulik11,migliorini11} is a description of the object $Rf_*(\QQ_\cJ)$, with the formula for $H^*(C^{[n]})$ then appearing as a corollary. 

In contrast, we restrict ourselves to the study of the single curve $C$. This paper grew out of an attempt to prove Proposition \ref{thm:MacdonaldsFormula} without the technology of perverse sheaves and the decomposition theorem. That such a proof should exist was suggested to us by Richard Thomas.

The approach we take is inspired by Nakajima's \cite{nakajima97} and Grojnowski's \cite{grojnowski96} construction of an action of an infinite-dimensional Heisenberg algebra on the homologies of the Hilbert schemes of a smooth surface. Both the definition of our operators and the strategy for proving their commutation relations are analogous to the corresponding parts of Nakajima's paper. The main technical contribution of this paper lies in defining the operators and proving the commutation relations in the context of the singular spaces $C^{[n]}$.

For a curve $C$ which is smooth over a quasi-projective smooth base variety $S$, Moonen and Polischuk \cite{moonen_algebraic_2010} have computed $\oplus_{n\ge 0} A_{*}(C^{[n]})$ in terms of $A_{*}(J)$, using a similar strategy to that of this paper.
Their computation holds in Chow groups with integral coefficients.
Specialising to $S = \Spec \CC$ and tensoring the Chow groups with $\QQ$, we recover the Chow version of Proposition \ref{thm:MacdonaldsFormula} for a smooth $C$.

\subsection{Outline of the paper}
\label{sec:outline}
The paper is laid out as follows. In Section \ref{sec:definitionofoperators} we give the precise definitions of the 4 operators. In Section \ref{sec:proofofmainthm} we assume the commutation relations of Theorem \ref{thm:MainTheorem} (\ref{thm:MainTheorem:CommRels}) and deduce parts (\ref{thm:MainTheorem:StructureResult}) and (\ref{thm:MainTheorem:WEqualsJ}). 

For the proof of the commutation relations, it will be convenient to use the language of bivariant homology theory, as laid out in \cite{fulton81}. In Section \ref{sec:bivariant} we give a summary of the relevant parts of this theory, and in Section \ref{sec:proofofcommutationrelations} we prove Theorem \ref{thm:MainTheorem} (\ref{thm:MainTheorem:CommRels}). In Section 6 we collect a few lemmas on the incidence schemes $C^{[n,n+1]}$ which we need elsewhere. Finally, in Section 7 we show that the grading $D$ is a splitting of the perverse filtration of \cite{maulik11, migliorini11}.

\subsection{Acknowledgements}
I would like to thank S.\ Kleiman, A.\ MacPherson, A.\ Oblomkov, R.\ Pandharipande, J.\ Rognes, V.\ Shende, C.\ Vial and my supervisor R.\ Thomas for useful discussions and comments related to this work.
Special thanks to C.\ Vial for pointing out and correcting an error in the proof of Thm.\ \ref{thm:MainTheorem} (\ref{thm:MainTheorem:WEqualsJ}).

\section{Definition of the four operators}
\label{sec:definitionofoperators}

\subsection{The deformation family of $C$}
\label{sec:versalfamily}
The following construction is essential for the definition of the operators $\mu_\pm[C]$ and for proving the commutation relations. 

Choose an algebraic family $f : \cC \to B$, where $B$ is nonsingular, such that $f^{-1}(0) \cong C$ for some $0 \in B$. Let $\cC^{[n]} \to B$ be the relative Hilbert scheme, that is the scheme such that the fibre over $b \in B$ is $(\cC_b)^{[n]}$. By \cite[Cor.\ 15]{shende12} we may choose the family so that the scheme $\cC^{[n]}$ is nonsingular for all $n$. Possibly after an étale base change, we may assume that the family admits a section $s : B \to \cC$ such that the image of $s$ is disjoint from the discriminant locus of $f$. Restricting the base further, we may assume that every curve in the family is reduced and irreducible.

For the remainder of the paper, we fix the data of the family $\cC \to B$, the section $s : B \to \cC$ and the nonsingular point $x = s(0) \in C$.

\subsection{Definition of $\mu_\pm[\pt]$}
\label{sec:definitionofeasyoperators}
Let $i : C^{[n]} \to C^{[n+1]}$ be the morphism defined on the level of points by
\[
\cI_{i(Z)} = \cI_x \cdot \cI_Z \ \ \ \ \ \forall Z \in C^{[n]}.
\]
In other words, the map $i$ is defined by adding a point at $x$.

\begin{nlemma}
\label{thm:embeddingIsRegular}
The embedding $i : C^{[n]} \into C^{[n+1]}$ is regular.
\end{nlemma}

\begin{proof}
The property of being regular is analytic local \cite[Lemma 2.6]{arbarello11}. Let $Z \in C^{[n]}$ be a point such that $Z$ has length $k$ at $x$. Choose an analytic open $U$ around $x$ such that the only component of $Z$ contained in $\ol{U}$ is the one at $x$. Then locally around $Z$ the morphism $i$ is isomorphic to 
\[
(U)^{[k]} \times (C\sm \ol{U})^{[n-k]} \stackrel{(j,\id)}{\into} (U)^{[k+1]} \times (C\sm \ol{U})^{[n-k]},
\]
where $j : (U)^{[k]} \into (U)^{[k+1]}$ is the morphism which adds a point at $x$. Since $(U)^{[k]}$ and $(U)^{[k+1]}$ are smooth, $j$ is a regular embedding, and hence so is $i$. 
\end{proof}
As a consequence of Lemma \ref{thm:embeddingIsRegular}, there is a Gysin map $i^! : H_*(C^{[n]}) \to H_{*-2}(C^{[n-1]})$. We let $\mpp = i_*$ and $\mmp = i^!$.

\subsection{Definition of $\mu_\pm[C]$}
\label{sec:DefinitionOfMuC}
The operators $\mu_\pm[C]$ are defined as correspondences in the following way. Let $C^{[n,n+1]} \subset C^{[n]} \times C^{[n+1]}$ be the flag Hilbert scheme parametrising pairs $(Z,Z^\pr)$ such that $Z \subset Z^\pr$. Let $\cC^{[n,n+1]}$ be its relative version, that is the scheme over $B$ such that for every $b \in B$, the fibre over $b$ is $(\cC_b)^{[n,n+1]}$. We then have the diagram
\[
\xymatrix{&C^{[n,n+1]} \ar@{->}[dl]_p \ar@{->}[dr]^q& \\
C^{[n]}&&C^{[n+1]}.
}
\]
We will define maps $p^! : H_*(C^{[n]}) \to H_{*+2}(C^{[n,n+1]})$ and $q^!:H_*(C^{[n+1]}) \to H_*(C^{[n,n+1]})$.
With these maps defined, we let $\mu_+[C] = q_*p^!$ and $\mu_-[C] = p_*q^!$. 

Consider the Cartesian square
\[
\xymatrix{C^{[n,n+1]} \ar@{->}[d]_p \ar@{^{(}->}[r] & \cC^{[n,n+1]} \ar@{->}[d]\\
C^{[n]} \ar@{^{(}->}[r] & \cC^{[n]}}
\]
Let $d = \dim \cC^{[n]}$.
By Lemma \ref{thm:FlagHilbertFamilyIsIrreducible}, $\cC^{[n,n+1]}$ is irreducible of dimension $d + 1$.

Since $\cC^{[n]}$ is nonsingular, we have $H_{*}(C^{[n]}) \cong H^{*}(\cC^{[n]}, \cC^{[n]} \setminus C^{[n]})$.
It then follows from \cite[Ex.\ 19.9.10]{fulton98} that there exists a refined intersection product 
\[
- \times - : H_{k}(C^{[n]}) \otimes H_{l}^{BM}(\cC^{[n,n+1]}) \to H_{k + l - 2d}(C^{[n,n+1]}).
\]

Now let $\alpha \in H_{k}(C^{[n]})$, and let $[\cC^{[n,n+1]}] \in H^{BM}_{2d+2}(\cC^{[n,n+1]})$ be the fundamental class.
We then define $p^{!}(\alpha) = \alpha \times [\cC^{[n,n+1]}] \in H_{k+2}(C^{[n,n+1]})$.
The definition of $q^{!}$ is similar.

\section{Proof of main theorem from commutation relations}
\label{sec:proofofmainthm}
In this section, we take the commutation relations of Theorem \ref{thm:MainTheorem} (\ref{thm:MainTheorem:CommRels}) for granted and show how parts (\ref{thm:MainTheorem:StructureResult}) and (\ref{thm:MainTheorem:WEqualsJ}) of the theorem follow from this.
Part (\ref{thm:MainTheorem:StructureResult}) is a formal consequence of the commutation relations and the fact that $\mu_{-}[\pt]$ and $\mu_{-}[C]$ are locally nilpotent.
\label{sec:linearalgebra}
\begin{nlemma}
\label{thm:abcommute}
Let $V$ be a vector space over a field $k$ with $\text{char}(k) = 0$, and let $\mu_-,\mu_+ \in \text{End}(V)$ satisfy $[\mu_-,\mu_+] = \id$. Assume further that for every $v \in V$ there is an integer $n \ge 0$ such that $\mu_-^n v = 0$. Then the natural map
\[
(\ker \mu_-) \otimes k[\mu_+] \to V
\]
is an isomorphism.
\end{nlemma}
\begin{proof}
Note first of all that if $v \in \ker \mu_-$, the commutation relation implies that $\mu_-\mu_+^nv = n\mu_+^{n-1}v$.

Let $\phi : (\ker \mu_-) \otimes k[\mu_+] \to V$ be the natural map. We first show that $\phi$ is injective. Suppose not, then there is some relation
\[
\sum_{i=0}^n \mu_+^i v_i = 0\ \ \ \ \ v_i \in \ker \mu_-
\]
with $v_n$ non-zero. Acting on this relation by $\mu_-^n$ and using the commutation relation gives $n!v_n = 0$, which is a contradiction.

We next show that $\phi$ is surjective. For any $v \in V$, we define the nilpotency of $v$ to be the smallest $n \ge 0$ such that $\mu_-^nv = 0$. Suppose $\phi$ is not surjective, and let $v \in V$ be an element of minimal nilpotency among those such that $v \not\in \im \phi$. The nilpotency of $\mu_-v$ is less than that of $v$, so we have $\mu_-v \in \im \phi$. Hence we have
\[
\mu_- v = \sum_{i=0}^n \mu_+^i v_i\ \ \ \ \ v_i \in \ker \mu_-.
\]
Now write
\begin{equation}
\label{eqn:RewritingV}
v = \sum_{i=0}^n \frac{1}{i+1} \mu_+^{i+1} v_i + v^\pr
\end{equation}
for some $v^\pr \in V$. Applying $\mu_-$ to \eqref{eqn:RewritingV} shows that $v^\pr \in \ker \mu_-$. The right hand side of \eqref{eqn:RewritingV} then clearly belongs to $\im \phi$, hence $v$ does.
\end{proof}

\begin{proof}[Proof of Theorem \ref{thm:MainTheorem} (\ref{thm:MainTheorem:StructureResult})]
Since $\mmp$ commutes with $\mmc$ and $\mpp$, the action of $\mmc$ and $\mpp$ preserves $\ker \mmp$. Applying Lemma \ref{thm:abcommute} with $V = \ker \mmp$, $\mu_- = \mmc$, and $\mu_+ = \mpp$, we see that the natural map
\[
W \otimes \QQ[\mpp] = (\ker \mmp \cap \ker \mmc) \otimes \QQ[\mpp] \to \ker \mmp
\]
is an isomorphism. Similarly we find that the map $\ker \mmp \otimes \QQ[\mpc] \to V(C)$ is an isomorphism. Combining these two isomorphisms and the fact that $\mpc$ and $\mpp$ commute gives the result.
\end{proof}

Let $g$ be the arithmetic genus of $C$.
\begin{nlemma}
\label{thm:ajisisomorphism}
The map
\[
AJ_* : \ker \mmp \cap H_*(C^{[n]}) \to H_*(J)
\]
is injective for any $n$, and is an isomorphism for $n \ge 2g$.
\end{nlemma}
\begin{proof}
Since the map $\mpp : \ker \mmp \cap H^{*}(C^{[n]}) \to \ker \mmp \cap H^{*}(C^{[n+1]})$ is injective by Theorem \ref{thm:MainTheorem} (\ref{thm:MainTheorem:StructureResult}) and $AJ_{*} = AJ_{*} \circ \mpp$, it suffices to prove the claim when $n \ge 2g$.

For $n \ge 2g-1$ the morphism $AJ : C^{[n]} \to J$ is a $\PP^{n-g}$-bundle \cite{altman_compactifying_1980}. Let $\omega = [i(C^{[n-1]})] \in H^2(C^{[n]})$, where $i$ is the inclusion map $i : C^{[n-1]} \into C^{[n]}$, and let $r = n-g$ be the fibre dimension of $C^{[n]} \to J$. The divisor $i(C^{[n-1]}) \subset C^{[n]}$ is a projective subbundle, hence we may express every $\alpha \in H_*(C^{[n]})$ uniquely as
\begin{equation}
\label{eqn:ProjectiveBundleFormula}
\alpha = \sum_{i=0}^r \omega^i \cap AJ^!(\alpha_i)\ \ \ \ \ \alpha_i \in H_*(J),
\end{equation}
where $AJ^!$ is the Gysin pull-back associated to a projective bundle. (See \cite[Thm.\ 3.3]{fulton98} for a proof of this in the case of Chow groups.) Note that we have $AJ_*(\alpha) = \alpha_r$.

We first prove injectivity of $AJ_*$. By part (\ref{thm:MainTheorem:StructureResult}) of the main theorem, $\mpp$ is injective. Hence $\ker \mmp = \ker (\mpp\mmp)$. By definition of the operators we have
\[
\mpp\mmp(\alpha) = i_*i^!(\alpha) = \omega \cap \alpha\ \ \ \ \ \forall \alpha \in H_*(C^{[n]}).
\]
Suppose $AJ_*(\alpha) = 0$ and $\mmp(\alpha) = 0$. Writing $\alpha$ as above this means $\alpha_r = 0$, and further that $\omega \cap \alpha = 0$. This implies $\alpha_i = 0$ for all $i$, hence $\alpha = 0$.

To prove surjectivity when $n \ge 2g$, we note first that $r = n - g \ge g = \dim J$. Let $0 \not= \beta \in H_k(J)$, and let $\alpha = \omega^{r+1}\cap AJ^!(\beta)$. 
Write $\alpha$ in terms of $\alpha_i$ as in \eqref{eqn:ProjectiveBundleFormula}. 
Since $\beta \not = 0$ we have $k \le 2\dim J \le 2r$, and then the homological degree of $\alpha_0$ is $k-2-2r \le -2$, so we have $\alpha_0 = 0$. We now take 
\[
\gamma = \omega^r \cap AJ^!(\beta) - \sum_{i=0}^{r-1} \omega^i \cap AJ^{!}(\alpha_{i+1}).
\]
We see that $AJ_*(\gamma) = \beta$ and $\mpp\mmp(\gamma) = \omega \cap \gamma = 0$, hence $\mmp(\gamma) = 0$.
\end{proof}

\begin{proof}[Proof of Theorem \ref{thm:MainTheorem} (\ref{thm:MainTheorem:WEqualsJ})]
The inclusion map $i : C^{[n]} \to C^{[n+1]}$ commutes with the Abel--Jacobi map, in the sense that $AJ \circ i = AJ$. It follows that $AJ_* = AJ_*\circ\mpp$.

We first show $AJ_* : W \to H_*(J)$ is surjective. Let $\alpha \in H_*(J)$. By Lemma \ref{thm:ajisisomorphism} there exists some class $\overline{\alpha} \in \ker \mmp$ such that $AJ_*(\overline{\alpha}) = \alpha$. But by Theorem \ref{thm:MainTheorem} (\ref{thm:MainTheorem:StructureResult}) we may write
\[
\overline{\alpha} = \sum_i \mpp^i \overline{\alpha_i}
\]
with $\ol{\alpha_i} \in W$, which implies $\alpha = AJ_*(\sum \overline{\alpha_i})$.
Using Theorem \ref{thm:MainTheorem} (\ref{thm:MainTheorem:StructureResult}) and the fact that $C^{[n]} \to J$ is a $\PP^{n-g}$-bundle, one checks that $\dim W = \dim H_{*}(J)$, hence $AJ_{*}$ is an isomorphism.

If we want to prove the version of Theorem \ref{thm:MainTheorem} (\ref{thm:MainTheorem:WEqualsJ}) for Chow groups, the dimensions of $W$ and $A_{*}(J)$ may be infinite.
In this case we can prove injectivity directly as follows.

Let $\alpha \in W$ be such that $AJ_*(\alpha) = 0$.
If $\alpha \in W \cap A_{*}(C^{[n]})$ for some $n$, then Lemma \ref{thm:ajisisomorphism} shows $\alpha = 0$.
If this is not the case, then we can write $\alpha = \sum_{i=m}^{n} \alpha_{i}$, with $\alpha_{i} \in A_{*}(C^{[i]}) \cap W$ and $\alpha_{m}, \alpha_{n} \not= 0$.
Let $\beta = \sum \mpp^{n-i}(\alpha_{i})$.
Then $\mmp(\beta) = 0$ and $AJ_{*}(\beta) = 0$, hence by Lemma \ref{thm:ajisisomorphism} we have $\beta = 0$.
But $\mmc(\beta) = \sum (n-i)\alpha_{i} \not= 0$, which gives a contradiction.
\end{proof}

\section{Bivariant homology formalism}
\label{sec:bivariant}
In order to be precise about which Gysin pull back maps we are using and what the compatibilities between them are, we use the formalism of bivariant homology as presented by Fulton and MacPherson in \cite{fulton81}. As the scope of the general theory is quite broad, we give here a recap of the parts of the theory we need. See \cite{fulton81} for the full story and in particular Section I.3 for details on the 
topological case.

\subsection{Description of the bivariant theory}
The bivariant Borel--Moore homology theory assigns to each map $f: X \to Y$ of reasonable\footnote{We require that $X$ and $Y$ can be written as closed subspaces of $\RR^n$ for some $n$; see \cite[I.3.1.1]{fulton81}.} topological spaces
a graded abelian group $H^*(X\stackrel{f}{\to} Y)$. The theory is equipped with 3 operations.
\begin{itemize}
\item \emph{Product:} Given maps $X\stackrel{f}{\to}Y$ and $Y\stackrel{g}{\to}Z$, there is a product homomorphism
\[
H^i(X\stackrel{f}{\to} Y) \otimes H^j(Y\stackrel{g}{\to}Z) \to H^{i+j}(X\stackrel{g \circ f}{\to}Z).
\]
For $\alpha \in H^i(X\stackrel{f}{\to} Y)$ and $\beta \in H^j(Y\stackrel{g}{\to}Z)$ we thus get a product $\alpha \cdot \beta \in H^{i+j}(X\stackrel{g \circ f}{\to}Z)$.
\item \emph{Pushforward:} For any proper map $X \stackrel{f}{\to} Y$ and any map $Y \stackrel{g}{\to} Z$ there is a pushforward homomorphism $f_* : H^*(X\stackrel{g\circ f}{\to} Z) \to H^*(Y\stackrel{g}{\to}Z)$.
\item \emph{Pullback:} For any Cartesian square
\[
\xymatrix{X^\pr \ar@{->}[r] \ar@{->}[d]^{g} & X \ar@{->}[d]^{f} \\
Y^\pr \ar@{->}[r] & Y
}
\]
there is a pullback homomorphism $H^*(X \stackrel{f}{\to} Y) \to H^*(X^\pr \stackrel{g}{\to} Y^\pr)$.
\end{itemize}
These operations satisfy various compatibility axioms, see \cite[Sec.\ I.2.2]{fulton81}.

\subsection{Relation to homology}
\label{sec:relationtohomology}
For any space $X$, the groups $H^i(X \to \pt)$ and $H^i(X \stackrel{\id}{\to} X)$ are identified with $H^{BM}_{-i}(X)$ and $H^i(X)$, respectively. Note that the three bivariant operations recover the usual homological operations of cup and cap product, proper pushforwards in homology and arbitrary pullbacks in cohomology.

\subsection{Nonsingular targets}
\label{sec:BivariantTheoryWithSmoothTarget}
The following observation will be crucial. If $Y$ is a nonsingular variety and $f: X \to Y$ is any morphism, the induced homomorphism
\[
H^*(X\stackrel{f}{\to} Y) \to H^{*-2\dim Y}(X\to \pt) = H^{BM}_{2\dim Y-*}(X)
\]
given by taking the product with $[Y] \in H^{-2\dim Y}(Y\to \pt)$ is an isomorphism. In such a situation we will frequently identify $H^{*}(X \to Y)$ with $H^{\mathrm{BM}}_{2\dim Y-*}(X)$.

In particular, if $X$ has a fundamental class $[X] \in H_{2\dim X}^{BM}(X)$, this induces a class $[X]\in H^{2(\dim Y - \dim X)}(X \to Y)$.
\subsection{Gysin maps}
Any class $\alpha \in H^i(X \stackrel{f}{\to} Y)$ defines a Gysin pull-back map $f^!:H^{BM}_*(Y) \to H_{*-i}^{BM}(Y)$ by
\[
f^!(\beta) = \alpha\cdot\beta, \ \ \ \ \ \forall \beta \in H^{\text{BM}}_*(Y).
\]

This relates to the Gysin maps $p^!$ and $q^!$ in the definition of $\mu_\pm[C]$ (see Sec.\ \ref{sec:DefinitionOfMuC}) as follows. Consider the Cartesian square
\[
\xymatrix{C^{[n,n+1]} \ar@{->}[d]_p \ar@{^{(}->}[r] & \cC^{[n,n+1]} \ar@{->}[d]\\
C^{[n]} \ar@{^{(}->}[r] & \cC^{[n]}}
\]
as in Sec.\ \ref{sec:DefinitionOfMuC}. The fundamental class $[\cC^{[n,n+1]}] \in H_*^{\mathrm{BM}}(\cC^{[n,n+1]})$ is identified with an element $[\cC^{[n,n+1]}] \in H^{-2}(\cC^{[n,n+1]} \to \cC^{[n]})$, since $\cC^{[n]}$ is nonsingular. Cartesian pullback defines an element $\wt{[\cC^{[n,n+1]}]} \in H^{-2}(C^{[n,n+1]} \to C^{[n]})$, and the Gysin pullback map associated with $\wt{[\cC^{[n,n+1]}]}$ coincides with $p^!$. A similar description can be given for $q^!$.

\subsection{Notation}
In a commutative diagram a Latin letter next to an arrow denotes the morphism, while a Greek letter denotes a bivariant homology class, so that e.g.\ the $\alpha$ in $X\xrightarrow[f]{\alpha}Y$ denotes an element $\alpha \in H^*(X\xrightarrow{f}Y)$.

\subsection{Chow theory}
There is a bivariant operational Chow theory assigning to every morphism of varieties $X \to Y$ an abelian group $A^*(X \to Y)$ \cite[Sec.\ I.9]{fulton81}. In this case $A^*(X \to \pt)$ equals the ordinary Chow group $A_*(X)$ of $X$. This bivariant theory is equipped with the same operations as the Borel--Moore theory satisfying the same compatibilities. It also has the property that $A^*(X \to Y) \stackrel{\cdot [Y]}{\to} A^{*-\dim Y}(X \to \pt)$ is an isomorphism for nonsingular $Y$ \cite[I.9.1.3]{fulton81}. Because of this, the proof of the commutation relations goes through verbatim upon replacing every $H$ with an $A$.

\section{Proof of commutation relations}
\label{sec:proofofcommutationrelations}
We now show that the operators obey the commutation relations of Theorem \ref{thm:MainTheorem} (\ref{thm:MainTheorem:CommRels}).
\subsection{Proof of $[\mmp, \mpc] = [\mmc,\mpp] = \id$}
\label{sec:umup}
Consider the diagrams
\[
\xymatrix{& {\cC^{[n,n+1]}} \ar@{->}[dl]^{p}_{\theta} \ar@{->}[d]_{\kappa}^{q} & X \ar@{_{(}->}[l]_{\wt{\iota}}^{\wt{i}} \ar@{->}[d]_{\wt{\kappa}}^{\wt{q}}\\
\cC^{[n]} & \cC^{[n+1]} & \cC^{[n]} \ar@{_{(}->}[l]_{\iota}^i
}
\]
and
\[
\xymatrix{& &\cC^{[n-1,n]} \ar@{->}[dl]_{\theta^\pr}^{p^\pr} \ar@{->}[dr]^{\kappa^\pr}_{q^\pr}& \\
\cC^{[n]} & \cC^{[n-1]} \ar@{_{(}->}[l]_{\iota^\pr}^{i^\pr} & &\cC^{[n]},}
\]
where in the first diagram $X = \cC^{[n,n+1]}\times_{\cC^{[n+1]}}\cC^{[n]}$ and the square containing $X$ is Cartesian. The morphisms $i,i^\pr$ correspond to adding a point at the section $s : B \to \cC$, see Sections \ref{sec:versalfamily}, \ref{sec:definitionofeasyoperators}. The bivariant classes $\theta, \iota, \kappa$ and their primed versions are the ones defined by fundamental classes, as in Section \ref{sec:BivariantTheoryWithSmoothTarget}. The classes $\wt{\iota}$ and $\wt{\kappa}$ are the Cartesian pullbacks of $\iota$ and $\kappa$, respectively.

Both of these diagrams are defined over the base $B$ of the family $\cC$. For any scheme, morphism or bivariant class we denote the result of performing the base change to $0 \in B$ by appending a subscript 0 to the object in question.

We first treat the case of $[\mmp,\mpc]$. For any $\alpha \in H_*(C^{[n]})$, we have
\[
\mmp\mpc(\alpha) = \iota_0 \cdot (q_0)_*(\theta_0\cdot \alpha) = (\wt{q}_0)_*(\wt{\iota}_0 \cdot \theta_0 \cdot \alpha)
\]
and
\[
\mpc\mmp(\alpha) = (q_0^\pr)_*(\theta_0^\pr\cdot\iota^\pr_0\cdot\alpha).
\]

\begin{nlemma}
\label{thm:BivariantToHomology}
Under the identification of $H^{*}(X \stackrel{p\circ \wt{i}}{\to} \cC^{[n]})$ with $H_{*+2\dim \cC^{[n]}}^{\mathrm{BM}}(X)$, we have
\[
\wt{\iota} \cdot \theta = [X]
\]
\end{nlemma}
\begin{proof}
The class $\wt{\iota}$ is the same as the class induced by $X \into \cC^{[n,n+1]}$ being the embedding of a Cartier divisor. It follows that
\[
\wt{\iota}\cdot\theta\cdot[\cC^{[n]}] = \wt{\iota}\cdot[\cC^{[n,n+1]}] = [X].
\]
\end{proof}

We will now compute $[X]$ by describing the irreducible components of $X$. In order to do this, we define certain maps $f : \cC^{[n-1,n]} \to X$ and $g : \cC^{[n]} \to X$. Since $X = \cC^{[n,n+1]} \times_{\cC^{[n+1]}} \cC^{[n]}$, we can describe $f$ and $g$ as products of suitable maps to $\cC^{[n,n+1]}$ and $\cC^{[n]}$.

We then let $f$ be the product of the map $\cC^{[n-1,n]} \to \cC^{[n,n+1]}$ sending $(Z,Z^\pr)$ to $(i(Z),i(Z^\pr))$ with the map $q^{\pr} : \cC^{[n-1,n]} \to \cC^{[n]}$. We let $g$ be the product of the map $\cC^{[n]} \to \cC^{[n,n+1]}$ sending $Z$ to $(Z,i(Z))$ with the identity map on $\cC^{[n]}$.

\begin{nlemma}
\label{thm:DecompositionOfX}
In $H_*^{\mathrm{BM}}(X)$ the equation
\[
[X] = f_*[\cC^{[n-1,n]}] + g_*[\cC^{[n]}]
\]
holds.
\end{nlemma}
\begin{proof}
It is easy to check on the level of points that $X = f(\cC^{[n-1,n]}) \cup g(\cC^{[n]})$, and that $f$ and $g$ are both injective. As $\cC^{[n]}$ and $\cC^{[n-1,n]}$ are both irreducible by Lemma \ref{thm:FlagHilbertFamilyIsIrreducible}, we get that $X = f(\cC^{[n-1,n]}) \cup g(\cC^{[n]})$ is the decomposition of $X$ into irreducible components. Furthermore, on the complement of $f(\cC^{[n-1,n]}) \cap g(\cC^{[n]})$ one checks that $f$ and $g$ are local isomorphisms of schemes. The claim follows.
\end{proof}

It follows that we have 
\[
\wt{\iota}\cdot\theta = [X] = f_*([\cC^{[n-1,n]}]) + g_*([\cC^{[n]}]) = f_*(\theta^\pr\cdot\iota^\pr) + g_*(1),
\]
where $1$ is the unit element in $H^0(\cC^{[n]} \stackrel{\id}{\to} \cC^{[n]})$. Using this we now compute
\begin{align}
\label{eqn:FinalComputation}
\mu_-[p]\mu_+[C](\alpha) &= (\wt{q}_0)_*(\wt{\iota}_0\cdot\theta_0\cdot\alpha) = (\wt{q}_0)_*(f_*(\theta^\pr\cdot\iota^\pr)_0\cdot\alpha) + (\wt{q}_0)_*(g_*(1)_0\cdot\alpha) \notag \\
&= (\wt{q}_0 \circ f_0)_*(\theta_0^\pr\cdot\iota_0^\pr\cdot\alpha) + (\wt{q}_0 \circ g_0)_*(\alpha) \notag \\
&= (q^\pr_0)_*(\theta_0^\pr\cdot\iota_0^\pr\cdot\alpha) + (\id)_*(\alpha) = \mu_+[C]\mu_-[p](\alpha) + \alpha,
\end{align}
which is what we wanted to show.

The proof of $[\mmc, \mpp] = \id$ is similar to the above case. Here we have
\[
\mu_-[C]\mu_+[\pt](\alpha) = (p_0)_*(\kappa_0\cdot(i_0)_*(\alpha)) = (p_0\circ \wt{i}_0)_*(\wt{\kappa}_0\cdot\alpha)
\]
and
\[
\mu_+[\pt]\mu_-[C](\alpha) = (i_0^\pr\circ p^\pr_0)_*(\kappa^\pr_0\cdot\alpha).
\]
Under the identification of $H^*(X \stackrel{\wt{q}}{\to} \cC^{[n]})$ with $H_{*+2\dim \cC^{[n]}}^{\mathrm{BM}}(X)$ we have $\wt{\kappa} = [X]$. This follows from
\[
\wt{\kappa}\cdot[\cC^{[n]}] = \wt{\kappa}\cdot\iota\cdot[\cC^{[n+1]}] = \wt{\iota}\cdot\kappa[\cC^{[n+1]}] = \wt{\iota}\cdot[\cC^{[n,n+1]}] = [X],
\]
where the last equality is obtained as in the proof of Lemma \ref{thm:BivariantToHomology}. Using Lemma \ref{thm:DecompositionOfX} we get
\[
\wt{\kappa} = [X] = f_*[\cC^{[n-1,n]}] + g_*[\cC^{[n]}] = f_*(\kappa^\pr) + g_*(1).
\]
A computation similar to \eqref{eqn:FinalComputation} now shows $\mu_-[C]\mu_+[\pt](\alpha) = \mu_+[\pt]\mu_-[C](\alpha) + \alpha$ as needed.

\subsection{Proof of $[\mpc, \mmc] = 0$}
The relevant diagrams are
\begin{equation}
\label{eqn:DiagramForX}
\xymatrix{
&&X \ar@{->}[dl]_{\wt{\kappa}} \ar@{->}[dr]_{\wt{q}}&& \\
&\cC^{[n,n+1]} \ar@{->}[dl]_\theta \ar@{->}[dr]_q&&\cC^{[n,n+1]} \ar@{->}[dl]_\kappa \ar@{->}[dr]_p& \\
\cC^{[n]} && \cC^{[n+1]} && \cC^{[n]}
}
\end{equation}
and
\begin{equation}
\label{eqn:DiagramForY}
\xymatrix{
&&Y \ar@{->}[dl]_{\wt{\theta}} \ar@{->}[dr]_{\wt{p}}&& \\
&\cC^{[n-1,n]} \ar@{->}[dl]_{\kappa^\pr} \ar@{->}[dr]_{p^\pr}&&\cC^{[n-1,n]} \ar@{->}[dl]_{\theta^\pr} \ar@{->}[dr]_{q^\pr}& \\
\cC^{[n]} && \cC^{[n-1]} && \cC^{[n]}.
}
\end{equation}
Here $X = \cC^{[n,n+1]} \times_{\cC^{[n+1]}} \cC^{[n,n+1]}$, $Y = \cC^{[n-1,n]} \times_{\cC^{[n-1]}} \cC^{[n-1,n]},$ and the squares containing $X$ and $Y$ are Cartesian. The bivariant classes $\theta, \kappa, \theta^\pr, \kappa^\pr$ are the ones induced by fundamental classes, while $\wt{\theta}$ and $\wt{\kappa}$ are the Cartesian pullbacks of $\theta^\pr$ and $\kappa$, respectively. 
As in Section \ref{sec:umup}, a base change to the central fibre $C = \cC_0$ is denoted by a subscript 0.

Let $\alpha \in H_*(C^{[n]})$. We then have
\[
\mmc\mpc(\alpha) = (p_0)_*(\kappa_0\cdot (q_0)_*(\theta_0\cdot\alpha)) = (p_0\circ\wt{q}_0)_*(\wt{\kappa}_0\cdot\theta_0\cdot\alpha)
\]
and
\[
\mpc\mmc(\alpha) = (q_{0}^\pr)_*(\theta^\pr_0\cdot(p^\pr_0)_*(\kappa_0^\pr\cdot\alpha)) = (q_0^\pr\circ\wt{p}_0)_*(\wt{\theta}_0\cdot\kappa_0^\pr\cdot\alpha).
\]

\begin{nlemma}
\label{thm:XAndYAreGenericallyReduced}
The scheme $X$ is equidimensional, and the scheme $Y$ is irreducible. Both are generically reduced, and
\[
\dim X = \dim Y = \dim \cC^{[n+1]}.
\]
\end{nlemma}
\begin{proof}
As $X = \cC^{[n,n+1]} \times_{\cC^{[n+1]}} \cC^{[n,n+1]}$, every irreducible component has dimension at least
\[
2 \dim \cC^{[n,n+1]} - \dim \cC^{[n+1]} = \dim \cC^{[n+1]}.
\]
Let $\Delta \subset B$ be the discriminant locus, i.e.\ the set of $b \in B$ such that $\cC_b$ is singular. By Lemma \ref{thm:EstimatesOfDimC} (\ref{thm:EstimatesOfDimC:FibreOverUp}) we have
\[
\dim X_\Delta = \dim \Delta + n + 1 \le (\dim B - 1) + n + 1 < \dim \cC^{[n+1]}.
\]
It follows that $X \sm X_\Delta$ is dense in $X$.

Write 
\[
X = \{(Z_1,Z_2,Z_3) \in \cC^{[n]} \times_B \cC^{[n+1]} \times_B \cC^{[n]} \mid Z_1, Z_3 \subset Z_2\}.
\]
Let $X_1 \subset X$ be the locus where $Z_1 = Z_3$, and let $X_2 = X \sm X_1$. It is then easy to check that $X_1 \cap (X\sm X_\Delta)$ and $X_2 \cap (X\sm X_\Delta)$ are irreducible, generically nonsingular, and of dimension equal to $\dim \cC^{[n+1]}$. This proves the claims for $X$.

Arguing similarly for $Y$, using Lemma \ref{thm:EstimatesOfDimC} (\ref{thm:EstimatesOfDimC:FibreOverDown}) we find that $Y \sm Y_\Delta$ is dense in $Y$. There is a morphism $\cC^{[n-1,n]} \to \cC$ taking a pair $(Z,Z^\pr)$ to the point where $Z$ and $Z^\pr$ differ. Using this we get a map
\[
Y = \cC^{[n-1,n]} \times_{\cC^{[n-1]}} \cC^{[n-1,n]} \to \cC \times_B \cC^{[n-1]} \times_B \cC.
\]
One checks that restricting both source and target to the locus of nonsingular curves this map is an isomorphism, hence $Y \sm Y_\Delta$ is isomorphic to 
\[
(\cC \times_B \cC^{[n-1]} \times_B \cC) \sm (\cC \times_B \cC^{[n-1]} \times_B \cC)_\Delta.
\]
In particular $Y \sm Y_\Delta$ is nonsingular and irreducible of dimension equal to $\dim \cC^{[n+1]}$, and the claims for $Y$ follow.
\end{proof}

Let $\pi : X \to \cC^{[n]}$ and $\pi^\pr : Y \to \cC^{[n]}$ be the natural maps going down the left hand side of diagrams \eqref{eqn:DiagramForX} and \eqref{eqn:DiagramForY}, respectively.
\begin{nlemma}
\label{thm:BivariantClassesAreFundamentalClasses}
Identifying $H(X \stackrel{\pi}{\to} \cC^{[n]})$ with $H_*^{BM}(X)$ gives $\wt{\kappa}\cdot\theta = [X]$. Identifying $H(Y \stackrel{\pi^\pr}{\to} \cC^{[n]})$ with $H_*^{BM}(Y)$ gives $\wt{\theta}\cdot\kappa^\pr = [Y]$.
\end{nlemma}
\begin{proof}
We treat the case of $X$; the case of $Y$ is similar. We must show that $\wt{\kappa}\cdot\theta\cdot[\cC^{[n]}] = [X]$, and as $\theta\cdot[\cC^{[n]}] = [\cC^{[n,n+1]}]$, it suffices to show $\wt{\kappa}\cdot[\cC^{[n,n+1]}] = [X]$. The class $\wt{\kappa}\cdot[\cC^{[n,n+1]}]$ can be identified with the refined intersection product
\[
[\cC^{[n,n+1]} \times \cC^{[n,n+1]}] \cap \Delta \in H^{BM}_*(X),
\]
where we intersect the classes inside $\cC^{[n+1]} \times \cC^{[n+1]}$, and $\Delta$ denotes the diagonal in this space. As $X$ is generically reduced, the intersection multiplicity at each component is 1, by \cite[Prop.\ 8.2]{fulton98}, and so this intersection product equals $[X]$.
\end{proof}

Let $f :X \to \cC^{[n]} \times_B \cC^{[n]}$ and $g : Y \to \cC^{[n]} \times_B \cC^{[n]}$ be the maps induced by composing down both sides of diagrams \eqref{eqn:DiagramForX} and \eqref{eqn:DiagramForY}, respectively.
\begin{nlemma}
\label{thm:PushforwardOfXIsPushforwardOfY}
In $H^{BM}_*(\cC^{[n]}\times_B\cC^{[n]})$, the equality
\[
f_*[X] = g_*[Y]
\]
holds.
\end{nlemma}
\begin{proof}
Let $X = \ol{X_1} \cup \ol{X_2}$ be the decomposition of $X$ into irreducible components, where $X_1$ and $X_2$ are as in the proof of Lemma \ref{thm:XAndYAreGenericallyReduced}. By definition of $X_1$ the image $f(\ol{X_1})$ is contained in the diagonal $\cC^{[n]} \subset \cC^{[n]}\times_B \cC^{[n]}$. Hence $\dim f(\ol{X_1}) < \dim \cC^{[n+1]} = \dim \ol{X_1}$, and so $f_*[\ol{X_1}] = 0$.

Let $U = \left(\cC^{[n]}\times_B\cC^{[n]}\right) \sm \cC^{[n]}$. We claim that over $U$ the maps $f|_{X_2}$ and $g$ are injective with the same image. To see this, note that if $(Z_1,Z_3) \in U$, then $Z_1 \not= Z_3$, and so
\begin{align*}
(Z_1,Z_3) \in f(X_2) &\Leftrightarrow (Z_1,Z_1\cup Z_3,Z_3) \in X_2 \Leftrightarrow l(Z_1 \cup Z_3) = n+1 \\
&\Leftrightarrow l(Z_1 \cap Z_3) = n-1 \Leftrightarrow (Z_1, Z_1 \cap Z_3, Z_3) \in Y \\
&\Leftrightarrow (Z_1, Z_3) \in g(Y).
\end{align*}
As both $X_2$ and $Y$ are generically reduced, it follows that $f_*[X] = f_*[\ol{X_2}] = g_*[Y]$.
\end{proof}

Let $\pi_1,\pi_2 : \cC^{[n]} \times_B \cC^{[n]} \to \cC^{[n]}$ be the projections. Combining Lemmas \ref{thm:BivariantClassesAreFundamentalClasses} and \ref{thm:PushforwardOfXIsPushforwardOfY} shows that in $H(\cC^{[n]}\times_B\cC^{[n]} \stackrel{\pi_1}{\to} \cC^{[n]})$ we have $f_*(\wt{\kappa}\cdot\theta) = g_*(\wt{\theta}\cdot \kappa^\pr)$. Let $\alpha \in H_*(C^{[n]})$, and compute
\begin{align*}
\mmc\mpc(\alpha) &= (p_0\circ \wt{q}_0)_*(\wt{\kappa}_0 \cdot \theta_0 \cdot \alpha) = ((\pi_2)_0 \circ f_0)_*(\wt{\kappa}_0 \cdot \theta_0 \cdot \alpha) \\
&= ((\pi_2)_0)_*((f_0)_*(\wt{\kappa}_0 \cdot \theta_0) \cdot \alpha) = ((\pi_2)_0)_*((g_0)_*(\wt{\theta}_0 \cdot \kappa_0^\pr) \cdot \alpha) \\
&= ((\pi_2)_0 \circ g_0)_*(\wt{\theta}_0 \cdot \kappa_0^\pr \cdot \alpha) = (q_0^\pr \circ \wt{p}_0)_*(\wt{\theta}_0 \cdot \kappa_0^\pr \cdot \alpha) \\
&= \mpc\mmc(\alpha),
\end{align*}
which is what we wanted.

\subsection{Proof of $[\mu_\pm[\pt], \mu_\pm[C]] = 0$}
\label{sec:umvm}
Consider the diagram
\[
\xymatrix{&C^{[n,n+1]} \ar@{->}[dl]_p \ar@{->}[dr]^q \ar@{^{(}->}[rr]^{i^{\pr\pr}} & & C^{[n+1,n+2]} \ar@{->}[dl]_{p^\pr} \ar@{->}[dr]^{q^\pr} & \\
C^{[n]} \ar@{->}[rr]_{i} & & C^{[n+1]} \ar@{->}[rr]_{i^\pr} & & C^{[n+2]}.}
\]
We have
\[
\mpp\mpc = {i^\pr}_*q_*p^! = q^\pr_*i^{\pr\pr}_*p^! = q^\pr_*(p^\pr)^!i_* = \mpc\mpp
\]
and
\[
\mmc\mmp = p_*q^!(i^\pr)^! = p_*(i^{\pr\pr})^!(q^\pr)^! = i^!p^\pr_*(q^\pr)^! = \mmp\mmc,
\]
where the required compatibilities are easily checked.

\subsection{Proof of $[\mmp, \mpp] = 0$}
\label{sec:umvp}
For $\alpha \in H_*(C^{[n]})$, we have
\[
\mpp\mmp(\alpha) = i_*(i^!(\alpha)) = [i(C^{[n-1]})] \cap \alpha,
\]
where $[i(C^{[n-1]})] \in H^2(C^{[n]})$ is the class of the Cartier divisor $C^{[n-1]}$. On the other hand,
\[
\mmp\mpp(\alpha) = i^!(i_*(\alpha)) = i^*[i(C^{[n]})] \cap \alpha.
\]
It thus suffices to show the equality $[i(C^{[n-1]})] = i^*[i(C^{[n]})]$ in $H^2(C^{[n]})$.

For any nonsingular point $y\in C$, let $i_y : C^{[n]} \to C^{[n+1]}$ be defined by adding a point at $y$, so that we have $i = i_x$ for our chosen point $x$. For any $y \not= x$ we have
\[
[i(C^{[n-1]})] = i_y^*[i(C^{[n]})],
\]
in $H^2(C^{[n]})$, which follows from the corresponding equality of Cartier divisors. As $i_y^* = i^*$ the claim follows.\footnote{For the case of Chow homology we need the equality $[i(C^{[n-1]})] = i^*[i(C^{[n]})]$ in $\Pic(C^{[n]})$. At the level of rational equivalence, it is no longer true that $i_y^* = i^*$, but the relation still holds by noting the equality of Cartier divisors $[i(C^{[n-1]})] = i_y^*[i(C^{[n]})]$ and then letting $y$ tend to $x$.}

\section{Flag Hilbert schemes}
\label{sec:FlagHilbertSchemes}
In this section we prove some dimension estimates for the flag Hilbert schemes $C^{[n,n+1]}$ and related schemes.

Let $H_n \subset (\AA^2)^{[n]}$ be the set of $Z \in (\AA^2)^{[n]}$ such that $Z$ is supported at $0 \in \AA^2$. Similarly, let $H_{n,n+1} \subset (\AA^2)^{[n,n+1]}$ be the set of pairs $(Z,Z^\pr) \in (\AA^2)^{[n]}\times(\AA^2)^{[n+1]}$ such that $Z \subset Z^\pr$ and both are supported at 0.
We follow the convention that $\dim \varnothing = -1$.
\begin{nlemma}
\label{thm:EstimateOfSingleDimH}
We have $\dim H_{n,n+1} = n$ for all $n \ge 0$.
\end{nlemma}
\begin{proof}
For any $Z \in H_n$, let $d_-(Z) = \dim \{Z^{\pr} \in H_{n-1} \mid Z^\pr \subset Z\}$ and let $d_+(Z) = \dim \{Z^\pr \in H_{n+1} \mid Z \subset Z^\pr\}$. 
We then have 
\begin{equation}
\label{eqn:dimUpAndDown}
d_+(Z) = d_-(Z) + 1
\end{equation}
for all $Z$, see \cite[Sec.\ 3]{EllingsrudStromme98}.

Let $V_{n,k} \subseteq H_n$ be the set of $Z\in H_n$ such that $d_-(Z) = k$.
Using \eqref{eqn:dimUpAndDown} we find

\begin{equation}
\label{eqn:DimHFlag}
\max_{k}\{\dim V_{n,k} + k +1\} = \dim H_{n,n+1} = \max_{k}\{\dim V_{n+1,k} + k\}.
\end{equation}
From \eqref{eqn:DimHFlag} we find $H_{n+1,n+2} = H_{n,n+1} + 1$, hence the claim of the lemma follows by induction from $\dim H_{0, 1} = 0$.
\end{proof}

\begin{nlemma}
\label{thm:EstimatesOfDimH}
For all $n \ge 0$, we have
\begin{enumerate}[(i)]
\item \label{thm:EstimatesOfDimH:FibreOverUp} $\dim H_{n,n+1} \times_{H_{n+1}} H_{n,n+1} = n$.
\item \label{thm:EstimatesOfDimH:FibreOverDown} $\dim H_{n,n+1} \times_{H_{n}} H_{n,n+1} = n+1$, unless $n = 0$, in which case $\dim H_{0, 1} \times_{H_{0}} H_{0,1} = 0$.
\end{enumerate}
\end{nlemma}

\begin{proof}
For $P = \ref{thm:EstimatesOfDimH:FibreOverUp}, \ref{thm:EstimatesOfDimH:FibreOverDown}$ and $n \ge 0$, let $(P)_n$ denote the claim that equation $(P)$ holds for the given value of $n$. We will prove the claims by induction, starting from the trivial cases $(\ref{thm:EstimatesOfDimH:FibreOverUp})_0$ and $(\ref{thm:EstimatesOfDimH:FibreOverDown})_0$. Let $X_n$ and $Y_n$ denote the schemes appearing on the left hand side of $(\ref{thm:EstimatesOfDimH:FibreOverUp})$ and $(\ref{thm:EstimatesOfDimH:FibreOverDown})$, respectively.

$(\ref{thm:EstimatesOfDimH:FibreOverDown})_{n-1} \implies (\ref{thm:EstimatesOfDimH:FibreOverUp})_n$: 
The diagonal map defines an inclusion $H_{n,n+1} \into X_n$, whence by Lemma \ref{thm:EstimateOfSingleDimH} we have $\dim X_{n} \ge n$, and it suffices to show that $\dim (X_n \sm H_{n,n+1}) \le n$.

The set of points of $X_n$ is
\[
\{(Z_1,Z_2,Z_3) \in H_n \times H_{n+1} \times H_n \mid Z_1,Z_3 \subset Z_2\},
\]
and $X_n \sm H_{n,n+1}$ is the locus of triples $(Z_1,Z_2,Z_3)$ where $Z_1 \not = Z_3$.
For such triples we must have $Z_2 = Z_1 \cup Z_3$.
Let $l(Z)$ denote the length of $Z$. 
Using the relation $l(Z_1 \cup Z_3) = l(Z_1) + l(Z_3) - l(Z_1 \cap Z_3)$ we get bijections
\begin{align*}
X_n \sm H_{n,n+1} &= \{(Z_1,Z_3) \in H_n \times H_n \mid l(Z_1 \cup Z_3) = n+1\} \\
&= \{(Z_1,Z_3) \in H_n \times H_n \mid l(Z_1 \cap Z_3) = n-1\} \\
&= \{Z_1, Z_1 \cap Z_3, Z_3\} \subseteq Y_{n-1},
\end{align*}
hence $\dim (X_n \sm H_{n,n+1}) \le \dim Y_{n-1} \le n$, by our assumption $(\ref{thm:EstimatesOfDimH:FibreOverDown})_{n-1}$.

$(\ref{thm:EstimatesOfDimH:FibreOverUp})_{n-1} \implies (\ref{thm:EstimatesOfDimH:FibreOverDown})_{n}$: Let $d_{-}, d_{+} : H_{n} \to \ZZ$ and $V_{n,k} = d_{-}^{-1}(k) \subseteq H_n$ be as in the proof of Lemma \ref{thm:EstimateOfSingleDimH}.
We write as above
\[
X_{n-1} = \{(Z_1,Z_2,Z_3) \in H_{n-1}\times H_{n} \times H_{n-1} \mid Z_1,Z_3 \subset Z_2\}.
\]
For any $Z \in H_{n}$, the fibre over $Z$ under the projection $X_{n-1} \to H_{n}$ has dimension $2d_-(Z)$. It follows that the locus in $X_{n-1}$ such that $Z_2 \in V_{n,k}$ has dimension $\dim V_{n,k} + 2k$.

Similarly
\[
Y_n = \{(Z_1^\pr,Z_2^\pr, Z_3^\pr) \in H_{n+1} \times H_{n} \times H_{n+1} \mid Z_2^\pr \subset Z_1^\pr, Z_3^\pr\},
\]
and the fibre over $Z \in H_{n}$ under the projection $Y_n \to H_{n}$ has dimension $2d_+(Z) = 2d_-(Z) + 2$, by \eqref{eqn:dimUpAndDown}. Hence the locus in $Y_n$ where $Z_2^\pr \in V_{n,k}$ has dimension $\dim V_{n,k} + 2k + 2.$

We get
\begin{align*}
\dim X_{n-1} = \max_k \{\dim V_{n,k} + 2k\} = \max_k\{ \dim V_{n,k} + 2k + 2\} - 2 = \dim Y_{n} - 2,
\end{align*}
hence by the induction assumption $(\ref{thm:EstimatesOfDimH:FibreOverUp})_{n-1}$ we get $\dim Y_{n} = \dim X_{n-1} + 2 = n+1$.
This concludes the induction procedure.
\end{proof}

\begin{nlemma}
\label{thm:EstimatesOfDimC}
Let $C$ be a locally planar reduced curve, and let $C_{\text{sm}} \subseteq C$ be its nonsingular locus.
For all $n \ge 0$, we have
\begin{enumerate}[(i)]
\item \label{thm:EstimatesOfDimC:PlainH} $\dim C^{[n,n+1]} = n + 1.$
\item \label{thm:EstimatesOfDimC:PlainHInEq} $\dim (C^{[n,n+1]} \setminus (C_{\text{sm}})^{{[n,n+1]}}) < n + 1.$
\item \label{thm:EstimatesOfDimC:FibreOverUp} $\dim C^{[n,n+1]} \times_{C^{[n+1]}} C^{[n,n+1]} = n + 1.$
\item \label{thm:EstimatesOfDimC:FibreOverDown} $\dim C^{[n,n+1]} \times_{C^{[n]}} C^{[n,n+1]} = n + 2.$
\end{enumerate}
\end{nlemma}
\begin{proof}
For points (\ref{thm:EstimatesOfDimC:PlainH}), (\ref{thm:EstimatesOfDimC:FibreOverUp}) and (\ref{thm:EstimatesOfDimC:FibreOverDown}) the claim LHS $\ge$ RHS is straightforward to see by replacing $C$ with $C_{\text{sm}}$, hence it suffices to prove the claim LHS $\le$ RHS.
We shall only prove the claim (\ref{thm:EstimatesOfDimC:FibreOverDown}); the other three claims can be handled by similar arguments.

Let $X = C^{[n,n+1]} \times_{C^{[n]}} C^{[n,n+1]}$, and write
\[
X = \{(Z_1,Z_2,Z_3) \in C^{[n+1]} \times C^{[n]} \times C^{[n+1]} \mid Z_2 \subset Z_1,Z_3\}.
\]
Let $\{x_1, \ldots, x_k\}$ be the set of singular points of $C$. We partition $X$ into disjoint subsets $X(a_0, \ldots, a_k, r, s)$, where the $a_i$ are non-negative integers whose sum is $n$, and where $r,s$ are integers such that $0 \le r,s \le k$. The subset $X(a_0, \ldots, a_k, r, s)$ parametrises $(Z_{1},Z_{2},Z_{3})$ satisfying the two conditions
\begin{enumerate}
\item $Z_2$ has support of length $a_0$ over the smooth locus of $C$ and of length $a_i$ at the point $x_i$ for $i > 0$.
\item The scheme $Z_1$ (resp.\ $Z_{3}$) differs from $Z_2$ at point $x_r$ if $r > 0$ (resp.\ $x_{s}$ if $s > 0$), and differs at a smooth point of $C$ if $r = 0$ (resp.\ $s = 0$).
\end{enumerate}

Let $C_{\text{sm}} \subset C$ be the nonsingular locus. Using the local planarity of $C$ we see that $X(a_0, \ldots, a_k, r, s)$ is isomorphic to a subset of one of the following schemes, depending on $r$ and $s$.
\begin{itemize}
\item $r = s = 0: ((C_{\text{sm}})^{[a_0,a_0+1]} \times_{(C_{\text{sm}})^{[a_0]}}(C_{\text{sm}})^{[a_0,a_0+1]}) \times H_{a_1} \times \cdots \times H_{a_k}$.
\item $r = s \not= 0: (C_{\text{sm}})^{[a_0]} \times H_{a_1} \times \ldots \times (H_{a_r,a_r+1} \times_{H_{a_r}} H_{a_r, a_r+1}) \times \ldots \times H_{a_k}$.
\item $r \not= s = 0: (C_{\text{sm}})^{[a_0,a_0+1]} \times H_{a_1} \times \ldots \times H_{a_r, a_r+1} \times \ldots \times H_{a_k}$.
\item $0 \not= r \not= s \not=0: (C_{\text{sm}})^{[a_0]} \times H_{a_1} \times \ldots \times H_{a_r, a_r+1} \times \ldots \times H_{a_s, a_s+1} \times \ldots \times H_{a_k}$.
\end{itemize}
Using Lemmas \ref{thm:EstimateOfSingleDimH}, \ref{thm:EstimatesOfDimH} and the fact that $\dim H_m = \max(0,m-1)$ (see \cite{iarrobino_punctual_1972}), we find that each of the above listed schemes has dimension $\le a_0 + \cdots + a_k + 2 = n + 1$.
The claim of the lemma follows.
\end{proof}

\begin{nlemma}
\label{thm:FlagHilbertSchemesAreIrreducible}
Let $C$ be a locally planar reduced curve. Then $(C_{\text{sm}})^{[n,n+1]}$ is dense in $C^{[n,n+1]}$.
\end{nlemma}
\begin{proof}
Kleiman and Altman \cite{AltmanKleiman79} have shown that we may embed $C$ in a nonsingular quasiprojective surface $S$. By work of Cheah and Tikhomirov \cite{Cheah98, Tikhomirov97}, the scheme $S^{[n,n+1]}$ is nonsingular of dimension $2n+2$. Because of the Cartesian diagram
\[
\xymatrix{C^{[n,n+1]} \ar@{->}[r] \ar@{->}[d] &S^{[n,n+1]} \ar@{->}[d]\\
C^{[n+1]} \ar@{->}[r] &S^{[n+1]},}
\]
every irreducible component of $C^{[n,n+1]}$ has dimension at least equal to
\[
\dim C^{[n+1]} + \dim S^{[n,n+1]} - \dim S^{[n+1]} = n+1.
\]
By Lemma \ref{thm:EstimatesOfDimC} (\ref{thm:EstimatesOfDimC:PlainHInEq}) we have $\dim (C^{[n,n+1]} \setminus (C_{\text{sm}})^{{[n,n+1]}}) < n + 1$.
It follows that the open subset $(C_{\text{sm}})^{[n,n+1]}$ intersects every irreducible component in $C^{[n,n+1]}$, hence it is dense as claimed.
\end{proof}

Let $\cC \to B$ be a family of curves satisfying the hypotheses of Section \ref{sec:versalfamily}, that is, $\cC^{[n]}$ is nonsingular and every curve in the family is irreducible and reduced.
\begin{nlemma}
\label{thm:FlagHilbertFamilyIsIrreducible}
The relative flag Hilbert scheme $\cC^{[n,n+1]}$ is irreducible of dimension equal to $\dim \cC^{[n+1]}$.
\end{nlemma}
\begin{proof}
Let $U \subseteq \cC$ be the locus of $q \in \cC$ such that $q \in \cC_b$ with $\cC_b$ smooth at $q$. As every curve in the family is irreducible, we get that $U^{[n,n+1]} \subset \cC^{[n,n+1]}$ is irreducible. Now for every fibre $\cC_b$, we have that $\cC_b^{[n,n+1]} \cap U^{[n,n+1]}$ is dense in $\cC_b^{[n,n+1]}$, by Lemma \ref{thm:FlagHilbertSchemesAreIrreducible}. Hence $U^{[n,n+1]}$ is dense in $\cC^{[n,n+1]}$, and the claim follows.
\end{proof}

Using the techniques of \cite{shende12} and the fact that for a smooth surface $S$ the variety $S^{[n,n+1]}$ is nonsingular (see \cite{Cheah98, Tikhomirov97}) one can show that $\cC^{[n,n+1]}$ is nonsingular. As we do not need this stronger statement, we omit the proof.

\section{The $D$-grading splits the perverse filtration}
\label{sec:GradingRefinesFiltration}
In this section we relate the $D$-grading on $H^*(J)$ to the perverse filtration appearing in \cite{maulik11, migliorini11}.
We use the notation of Section \ref{sec:remarkoncohomology}, namely $\vcc = \oplus_{n\ge 0}H^*(C^{[n]})$, the operators $\mu^{\mathrm{c}}_\pm[\pt]$ and $\mu^{\mathrm{c}}_\pm[C]$ act on $\vcc$, and $\wc = \vcc/(\im \mcpp + \im \mcpc)$. The space $\vcc$ is equipped with a grading $D$, by letting
\[
D_n\vcc = H^*(C^{[n]}),
\]
and the spaces $\wc$ and $H^*(J)$ inherit this grading using the isomorphisms of Theorem \ref{thm:MainTheoremCohomology}. Recall also the formula
\begin{equation}
\label{MacdonaldFormulaCohomology2}
H^*(C^{[n]}) \cong \bigoplus_{m\le n}D_mH^*(J)\otimes\Sym^{n-m}\left(\QQ \oplus \QQ[-2]\right).
\end{equation}

We fix as usual the versal deformation family $\cC \to B$ and let $0 \in B$ be the point such that $\cC_0 = C$. Following \cite{maulik11}, we define the \emph{perverse filtration} $P_{\le j}$ on $H^*(J)$ as follows. The versal deformation family $f: \cC \to B$ induces a family of compactified Jacobians ${f_{\cJ}}:\cJ \to B$. The object $R{f_{\cJ}}_*\QQ_\cJ\in D^b_c(B)$ has a filtration $\tau^p_{\le j}R{f_{\cJ}}_*\QQ_\cJ$ induced by the perverse t-structure on $D^b_c(B)$. 
We have ${f_{\cJ}}^{-1}(0) = J$, and so if $g$ is the inclusion $0 \into B$, we naturally have $g^*(R{f_{\cJ}}_*\QQ_\cJ) = H^*(J)$. We may now define the perverse filtration by
\[
P_{\le j}H^*(J) = \mathrm{Im}\left(g^*\left(\tau_{\le j}^pR{f_{\cJ}}_*\QQ_\cJ\right) \to g^*\left(R{f_{\cJ}}_*\QQ_\cJ\right) = H^*(J)\right).
\]
Replacing $J$ and $\cJ$ with $C^{[n]}$ and $\cC^{[n]}$ in the above construction we get a filtration $P_{\le j}$ on $H^*(C^{[n]})$ as well.

For $X = J$ or $X = C^{[n]}$ we normalise the indices of the perverse filtration by letting $P_{\le -1}H^*(X) = 0$ and letting $1\in H^0(X)$ be contained in $P_{\le 0}H^*(X)$. It follows that $\gr^P_i(H^*(X)) = 0$ unless $0\le i \le 2\dim X$.

Comparing the formula of \cite[Thm.\ 1.1]{maulik11} with \eqref{MacdonaldFormulaCohomology2} we find an isomorphism $\gr_\bullet^PH^*(J) \cong D_\bullet H^*(J)$ of bigraded vector spaces. In other words the filtrations $P_{\le n}$ and $D_{\le n}$ on $H^*(J)$ have isomorphic associated graded objects. The remainder of this section is devoted to showing that the filtrations are in fact equal.

\begin{nprop}
\label{thm:DEqualsP}
$D_{\le n}H^*(J) = P_{\le n}H^*(J)$.
\end{nprop}

Let $X = C^{[n]}$ or $X = J$. We define the filtration $Q_{\le j}$ on $H^*(X)$ by
\[
Q_{\le j}H^i(X) = P_{\le i+j}H^i(X).
\]

\begin{nlemma}
The maps $\mu^{\mathrm{c}}_\pm[\pt]$, $\mu^{\mathrm{c}}_\pm[C]$ and $AJ^*$ all preserve the $Q$-filtration.
\end{nlemma}
\begin{proof}
The statement follows from the fact that each of the maps is the restriction to 0 of a map of complexes on $B$. We will give the details for $\mcpp$ and $\mcpc$; the remaining cases are similar and left to the reader.

For $\mcpp$, we have the following diagram
\[
\xymatrix{\cC^{[n]} \ar@{^{(}->}[rr]^i \ar@{->}[dr]_{f^{[n]}} && \cC^{[n+1]} \ar@{->}[dl]^{f^{[n+1]}}\\
&B&.
}
\]
Since $\cC^{[n]}$ and $\cC^{[n+1]}$ are nonsingular, we have $i^!(\QQ_{\cC^{[n+1]}}) = \QQ_{\cC^{[n]}}[-2]$. By adjunction we get a map $i_*\QQ_{\cC^{[n]}} \to \QQ_{\cC^{[n+1]}}[2]$, which we push down to get a map
\[
R{f_{*}^{[n]}}\QQ_{\cC^{[n]}} = R(f^{[n+1]}i)_*\QQ_{\cC^{[n]}} \to Rf^{[n+1]}_*\QQ_{\cC^{[n+1]}}[2].
\]
One can check that the restriction of this map to $0 \in B$ agrees with $\mcpp$.

Now, the composed map
\[
\tau^p_{\le j}R{f_{*}^{[n]}}\QQ_{\cC^{[n]}} \to R{f_{*}^{[n]}}\QQ_{\cC^{[n]}} \to Rf^{[n+1]}_*\QQ_{\cC^{[n+1]}}[2]
\]
factors through $\tau_{\le j}^p(Rf^{[n+1]}_*\QQ_{\cC^{[n+1]}}[2]) \to Rf^{[n+1]}_*\QQ_{\cC^{[n+1]}}[2].$ It follows that $\mcpp$ sends $P_{\le j}H^*(\cC^{[n]})$ to $P_{\le j+2}H^{*+2}(\cC^{[n+1]})$, which is the same as saying $\mcpp$ preserves the $Q$-filtration.

For the case of $\mcpc$, we have the diagram
\[
\xymatrix{&\cC^{[n,n+1]} \ar@{->}[dl]_p \ar@{->}[dr]^q& \\
\cC^{[n]} \ar@{->}[dr]_{f^{[n]}} && \cC^{[n+1]} \ar@{->}[dl]^{f^{[n+1]}} \\
&B&.
}
\]
Using the identification of $H_*^{\mathrm{BM}}(\cC^{[n,n+1]})$ with 
\[
R^{-*}\Gamma(\QQ_{\cC^{[n,n+1]}}^\vee) = \Hom_{D(\cC^{[n,n+1]})}(\QQ_{\cC^{[n,n+1]}},\QQ_{\cC^{[n,n+1]}}^\vee[-*]),
\] 
the fundamental class of $\cC^{[n,n+1]}$ induces a map 
\[
\QQ_{\cC^{[n,n+1]}} \to \QQ^\vee_{\cC^{[n,n+1]}}[-2\dim \cC^{[n,n+1]}] = q^!\QQ_{\cC^{[n+1]}}.
\]
Hence we get a map $Rq_*\QQ_{\cC^{[n,n+1]}} \to \QQ_{\cC^{[n+1]}}$, and thus a composed map
\begin{align*}
R{f^{[n]}_{*}}\QQ_{\cC^{[n]}} \to R({f^{[n]}}p)_*\QQ_{\cC^{[n,n+1]}} = R(f^{[n+1]}q)_*\QQ_{\cC^{[n,n+1]}} \to Rf^{[n+1]}_*\QQ_{\cC^{[n+1]}}.
\end{align*}
Again one can check that the restriction of this map to $0 \in B$ equals $\mcpc$. By the same argument as for $\mcpp$, we then get that $\mcpc$ preserves the $Q$-filtration.
\end{proof}

As a consequence of the above lemma, the operators $\mu_\pm^{\mathrm{c}}[\pt]$ and $\mu_\pm^{\mathrm{c}}[C]$ act on $\gr^Q\vcc$. Since they still obey the Weyl algebra commutation relations when acting on this space, the proof of the implication (\ref{thm:MainTheorem:CommRels}) $\Rightarrow$ (\ref{thm:MainTheorem:StructureResult}) in Theorem \ref{thm:MainTheorem} applies to show that
\begin{equation}
\label{eqn:GradedRepIso}
\gr^Q\wc\otimes\QQ[\mcpp,\mcpc] \cong \gr^Q\vcc.
\end{equation}
This is an isomorphism of $(H,Q,D)$-graded spaces, where $\mcpp$ and $\mcpc$ have degrees $(2,0,1)$ and $(0,0,1)$, respectively.

Let $D_nH^i\wc$ be the image of $H^i(C^{[n]})$ under the map $\vcc \to \wc$. Define three generating functions
\begin{align*}
F_V(x,y,z) &= \sum_{i,j,n}\dim \gr^Q_jH^i(C^{[n]})x^iy^jz^n \\
F_W(x,y,z) &= \sum_{i,j,n}\dim \gr^Q_jD_nH^i\wc x^iy^jz^n \\
F_J(x,y,z) &= \sum_{i,j}\dim \gr^Q_jH^i(J) x^iy^jz^{i+j}.
\end{align*}
The isomorphism \eqref{eqn:GradedRepIso} implies
\[
F_W \cdot (1-z)^{-1}(1-x^2z)^{-1} = F_V.
\]
It follows from \cite[Cor.\ 2]{migliorini11} that
\[
F_J \cdot (1-z)^{-1}(1-x^2z)^{-1} = F_V,
\]
and hence $F_J = F_W$. As a consequence, the coefficient of $x^iy^jz^n$ in $F_W$ is 0 unless $n = i + j$, and it follows that 
\begin{equation}
\label{eqn:FiltrationsOnWAreEqual}
D_{\le i + j}H^i\wc = Q_{\le j}H^i\wc.
\end{equation} 

From $F_J = F_W$ we get that $\dim \gr^Q_jH^*(J) = \dim \gr^Q_j\wc$ for all $j$. Since $AJ^* : H^*(J) \to \wc$ is an isomorphism and preserves the $Q$-filtration, it must then be an isomorphism of $Q$-filtered spaces. Obviously $AJ^*$ preserves the cohomological grading. The $D$-grading on $H^*(J)$ is defined via $AJ^*$ and so preserved by definition, hence by \eqref{eqn:FiltrationsOnWAreEqual} we have
\[
D_{\le i + j}H^i(J) = Q_{\le j}H^i(J) = P_{\le i+j}H^i(J),
\]
and the proof of Proposition \ref{thm:DEqualsP} is complete.

\bibliographystyle{alpha-abbrv}
\bibliography{bibliography}

%\begin{thebibliography}{9}
%\bibitem{fmp81}
%Fulton MacPherson, Categorical Framework for the study of singular spaces
%\end{thebibliography}
\end{document}